\documentclass[reqno,11pt]{amsart}

\RequirePackage{mathtools}

\numberwithin{equation}{section}

\usepackage[pdftex,hyperfootnotes]{hyperref}

\advance\textheight by \topskip
 
\usepackage{amsmath,amssymb,amsthm,mathrsfs,
}

\usepackage{longtable}
\usepackage{enumerate}
\usepackage{graphicx}
\usepackage{wasysym}
\usepackage{abraces}
\usepackage{accents}
\newcommand*{\dt}[1]{
 \accentset{\mbox{\large\bfseries .}}{#1}}

\usepackage{xcolor}

\allowdisplaybreaks

\theoremstyle{plain}
\newtheorem{theorem}{Theorem}[section]
\newtheorem{lemma}[theorem]{Lemma}
\newtheorem{proposition}[theorem]{Proposition}

\newtheorem{definition}[theorem]{Definition}

\newcommand{\interleave}{\|{\hspace{-.039cm}|}}

\newcommand{\T}{\top}

\DeclareMathOperator{\oA}{A}
\newcommand{\pA}{\pmb{\oA}}
\DeclareMathOperator{\oB}{B}
\DeclareMathOperator{\oC}{C}
\DeclareMathOperator{\oD}{D}

\DeclareMathOperator{\dd}{d}
\DeclareMathOperator{\oG}{G}
\DeclareMathOperator{\oH}{H}
\newcommand{\pH}{\pmb{\oH}}
\DeclareMathOperator{\oI}{I}
\DeclareMathOperator{\oJ}{J}
\newcommand{\pJ}{\pmb{\oJ}}
\DeclareMathOperator{\oL}{L}
\newcommand{\pL}{\pmb{\oL}}
\DeclareMathOperator{\oM}{M}

\DeclareMathOperator{\oP}{P}
\newcommand{\pP}{\pmb{\oP}}
\newcommand{\mP}{\mathbb{P}}

\DeclareMathOperator{\oS}{S}

\DeclareMathOperator{\rank}{rank}

\begin{document}

\title[Lax-type pairs in the theory of bivariate OP]{Lax-type pairs in the theory of bivariate orthogonal polynomials}

\author[A. 
Branquinho, 
A. 
Foulqui\'e-Moreno,
T. E.
P\'erez,
and M. A. 
Pi\~nar]
{Am\'ilcar Branquinho, Ana Foulqui\'e-Moreno, Teresa E. P\'erez, \\
and Miguel A. Pi\~nar}

\address[A. Branquinho]{CMUC,
Department of Ma\-the\-ma\-tics, University of Coimbra, Apartado 3008, EC Santa Cruz, 3001-501 COIMBRA, Portugal.}
\email{ajplb@mat.uc.pt}

\address[A. Foulqui\'e-Moreno]{CIDMA,
Departamento de Matem\'atica, Universidade de Aveiro, 3810-193 Aveiro, Portugal}
\email{foulquie@ua.pt}

\address[T. E. P\'erez]{Instituto 
de Matem\'aticas IMAG \&
Departamento de Ma\-te\-m\'{a}\-ti\-ca Aplicada, Facultad de Ciencias. Universidad de Granada (Spain)}
\email{tperez@ugr.es}

\address[M. A. Pi\~nar]{Instituto de Matem\'aticas IMAG \&
Departamento de Ma\-te\-m\'{a}\-ti\-ca Aplicada, Facultad de Ciencias. Universidad de Granada (Spain)}
\email{mpinar@ugr.es}

\thanks{AB acknowledges Centro de Matem\'atica da Universidade de Coimbra (CMUC){\textemdash}UIDB/00324/2020 (funded by the Portuguese Government through FCT/MCTES). \\
AFM acknowledge CIDMA{\textemdash}Center for Research \& Development in Mathematics and Applications is supported through the Portuguese Foundation for Science and Technology (FCT{\textemdash}Funda\c c\~ao para a Ci\^encia e a Tecnologia), references UIDB/04106/2020 and UIDP/04106/2020. \\
TEP and MAP thanks Grant FQM-246-UGR20 funded by Consejer\'ia de Universidad, Investigaci\'on e Innovaci\'on de la Junta de Andaluc\'ia and FEDER, Una manera de Hacer Europa; IMAG-Mar\'ia de Maeztu grant CEX2020-001105-M, and Research Group Goya FQM-384}

\date{\today}

\begin{abstract}
Sequences of bivariate orthogonal polynomials written as vector polynomials of increasing size satisfy a couple of three term relations with matrix coefficients. In this work, introducing a time-dependent parameter, we analyse a Lax-type pair system for the coefficients of the three term relations. We also deduce several characterizations relating the Lax-type pair, the shape of the weight, Stieltjes function, moments, a differential equation for the weight, and the bidimensional Toda-type systems. 
\end{abstract}

\subjclass[2020]{Primary: 42C05; 33C50; 35Q53}

\keywords{Two variable orthogonal polynomials, $2$D Toda lattice,
Block Lax pairs}

\maketitle

\section{Introduction}

Several areas of Classical Analysis have profound connections with the theory of orthogonal polynomials, such as moment problems, spectral theory of Jacobi matrices, and random matrices.

One of the more recent connection is with the integrable systems that starts in the 1970's with the work 
by Gardner, Greene, Kruskal and Miura in \cite{ggkm1,ggkm2},
with a method for the exact solution of the initial-value problem for the KdV equation, which is now referred to as the inverse scattering transform~\cite{Ablowitz_Segur}.
In~\cite{Lax}, Lax put the inverse scattering method for solving the~KdV equation into a more general framework, which subsequently paved the way to generalizations of the technique, as a method for solving other partial differential equations.

In fact, he considered, associated with the KdV equation, two time dependent operators $\mathcal L$ and $\mathcal M$, where $\mathcal L$ is the operator of the spectral problem and $\mathcal M$ is the operator governing the associated time evolution of the eigenfunctions
\begin{align*}
\mathcal L v = \lambda v, && \dfrac{\dd}{\dd t} v = \mathcal M v, 
\end{align*}
and hence
\begin{align*}
\dfrac{\dd}{\dd t} \mathcal L = \mathcal M \, \mathcal L - \mathcal L \, \mathcal M ,
 &&
\text{if, and only if,}
 &&
\dfrac{\dd}{\dd t} \lambda = 0 .
\end{align*}
 The relation between Toda lattices and univariate orthogonality has been analysed by several authors, for an introduction and first properties,~\cite{ABM,Is05, Na04, Pe01}, among many others, can be consulted. In the multivariate case, some papers studied the case when several continuous time dependent variables are considered, see \cite{ AM16, AM22}, and in \cite{ADMV16} multidimensional analogues of continuous and discrete-time Toda lattices with two or more space coordinates relating them with multiple orthogonal polynomials are studied.

The~$2$D~Toda lattices with only a time dependent variable can be seen in~\cite{BP17}.
There the authors depart from the knowledge of the weight representation and derive some Toda like lattices and at the end a Lax type representation for them.
In this paper we apply the method of moments, presented in~\cite{Aptekarev_AB} for the univariate Toda lattice, to a  theory so far general of $2$D Toda lattice.

Usually, orthogonal polynomials in two variables are written as vector polynomials of increasing size, and they satisfy three term relations with matrix coefficients (see \cite{DX14}).
We consider orthogonality with respect to a bivariate weight function and introduce a time-dependent parameter such that we work with an evolution weight function. Here we depart from the two tridiagonal block Jacobi operators and present the correspondence between dynamics of the $2$D Toda equations for the block-coefficients of the operators and the Stieltjes function associated with the block Jacobi operators.
This gives a method to solve an inverse problem, 
i.e., we can get an integral representation for the coefficients of the block Jacobi matrices that satisfies a $2$D Toda equations in terms of a weight function completely determined by the~data.

The work is organized in four sections. After this introduction, in Section~\ref{sec:2} we briefly present the theory and basic facts of bivariate orthogonal polynomials that we will need in the sequel.
In Section~\ref{sec:3} we give, in Theorem~\ref{teo:main}, an interpretation of the $2$D Toda lattice in terms of the theory of bivariate orthogonal polynomials. This is in fact the main results of the paper, where we characterize the $2$D Toda lattice in terms of the moments and the Stieltjes function associated with the $2$D Toda lattice.
There we also give a representation for the weight that governs the $2$D Toda lattice, and deduce several characterizations relating the Lax-type pair, the moments, a differential equation for the weight, and the bidimensional Toda-type system.
Finally, in Section~\ref{sec:4} we end the work by establishing a Lax type theorem and proving the isospectrality of the associated block Jacobi matrices.


\section{Preliminary results}\label{sec:2}

For each $n\geqslant 0$, let $\Pi_n$ denote the linear space of bivariate polynomials of total degree not greater than $n$ (cf.~\cite{DX14}).
We consider $\Pi=\bigcup_{n\geqslant 0} \Pi_n$ the linear space of all bivariate polynomials with real coefficients.

We say that $p(x,y)\in\Pi_n$ is \emph{monomial} if there is only one term of higher degree, i.e., there exists
$0\leqslant k\leqslant n$ such~that
\begin{align*}
p(x,y) = a_{n-k,k} x^{n-k}y^{k} + \sum_{m=0}^{n-1}\sum_{i=0}^m a_{m-i,i} x^{m-i}y^{i},
\end{align*}
with $a_{n-k,k}\neq 0$. Moreover, if $a_{n-k,k} = 1$, we 
say that $p(x,y)$ is \emph{monic}.

Let us denote by $\mathcal{M}_{h \times k}(\mathbb{R})$ the linear space of matrices of size $h\times k$ with
real entries and by $\mathcal{M}_{h \times k}(\Pi)$ the linear space of $h\times k$ matrices with polynomials in
two variables entries. The \emph{degree of a matrix polynomial} is defined as the maximum of the degrees of its
polynomial~entries.

Given a matrix $\oM \in \mathcal{M}_{h\times k}$, we denote by $\oM^{\top}$ its transpose. If $h=k$, we will denote $\mathcal{M}_{h\times h}\equiv \mathcal{M}_h$. In particular, $\oI_h$ denotes the identity matrix of size~$h$. 
We say that $\oM \in \mathcal{M}_{h}(\mathbb{R})$ is \emph{non-singular} if $\det \oM\neq 0$.

\subsection{Vector notation}

For each $n\geqslant 0$, let $\mathbb{X}_n$ denote the column vector
\begin{align*}
\mathbb{X}_n=
\begin{bmatrix}
x^n & x^{n-1} y & \cdots & xy^{n-1} & y^n
\end{bmatrix}^{\top},
\end{align*}
of size $(n+1)\times 1$. Then $\big\{\mathbb{X}_n\big\}_{n\geqslant 0}$ is called the \emph{canonical basis} of $\Pi$ and
every polynomial $\oP \in\Pi$ of degree $n$ can be represented~as
\begin{align*}
\oP(x,y)=\sum_{k=0}^n \oC_k^\top \, \mathbb{X}_k,
\end{align*}
where 
$ \oC_k$ is a $(k+1)\times 1$ vector of constants.
We continue, following the notation in~\cite{DX14}, for $n \geqslant 0$, denoting $\oL_{n,1}$ and $\oL_{n,2}$, as the matrices of size $(n + 1) \times
(n + 2)$, such~that
\begin{align*}
\oL_{n,1}\,\mathbb{X}_{n+1}=x\,\mathbb{X}_{n} && \textnormal{and} &&
\oL_{n,2}\,\mathbb{X}_{n+1}=y\,\mathbb{X}_{n},
\end{align*}
where, $\oL_{n,1}$ and $\oL_{n,2}$ represent the so-called shift operators associated with the multiplication of
$\mathbb{X}_n$ by the variables $x$ and $y$, respectively. Therefore,
\begin{align}\label{L_ni}
\oL_{n,1}=
\left[\begin{array}{ccc|c}
1 & {} & {} & 0 \\[-.125cm]
{} & \ddots & {} & \vdots \\
{} & {} & 1 & 0
\end{array}\right], 
 &&
\oL_{n,2}=\left[\begin{array}{c|ccc}
0 & 1 & {} & {}\\[-.125cm]
\vdots & {} & \ddots & {} \\
0 & {} & {} & 1
\end{array}\right].
\end{align}
Observe that, $\rank \oL_{n,i}=n+1$,
$\oL_{n,i} \oL_{n,i}^{\top}
= \oI_{n+1}$, 
for $i=1,2$, and $\oL_{n,1} \oL_{n+1,2} $ $ = \oL_{n,2} \oL_{n+1,1}$. 

Let $\big\{ \oP_{n,m}(x,y): 0\leqslant m \leqslant n, n\geqslant 0\big\}$ denote a basis of $\Pi$ such that, for a 
fixed~$n\geqslant 0$,
$\deg \oP_{n,m}(x,y) = n$, and the set $\big\{\oP_{n,m}(x,y): 0\leqslant m \leqslant n\big\}$ contains $n+1$ 
linearly independent polynomials of total degree exactly~$n$.
We can write the vector of polynomials
\begin{align*}
\mP _n= \begin{bmatrix}
\oP_{n,0}(x,y) & \oP_{n,1}(x,y) & \cdots & \oP_{n,n}(x,y)
\end{bmatrix}^{\top}.
\end{align*}
The sequence of polynomial vectors of increasing size
$\big\{\mP _n\big\}_{n \geqslant 0 }$ is called a \emph{polynomial system}~(PS), and it is a basis of $\Pi$.
Moreover, for each $n\geqslant 0$, the vector $\mP _n$ can be written as
\begin{align*}
\mP _n = \oG_{n}^n\,\mathbb{X}_n + \oG^{n}_{n-1}\,\mathbb{X}_{n-1} + \cdots + 
\oG^{n}_0\,\mathbb{X}_0,
\end{align*}
where $\oG_{n}^n$ is a $n+1$ non-singular matrix of constants, called \emph{the matrix leading coefficient} of~$\mP _n$, and $\oG^n_m$ are $(n+1)\times (m+1)$ constant matrices.
When $\oG_{n}^n$ is the identity matrix, i.e., $\oG_{n}^n = \oI_{n+1}$, for every $n\geqslant 0$, then we say that 
$\big\{\mP _n\big\}_{n\geqslant 0}$ is a \emph{monic PS}. Observe that the system of polynomials given by 
$\widehat{\mP }_n = (\oG_n^n)^{-1}\,\mP _n$, for $n \geqslant 0 $, is a monic polynomial system, in the sense~that
\begin{align*}
\widehat{\mP }_n = \mathbb{X}_n + \widehat{\oG}^{n}_{n-1}\,\mathbb{X}_{n-1} + 
\cdots + \widehat{\oG}^{n}_0\,\mathbb{X}_0,
\end{align*}
where $ \widehat{\oG}^{n}_{m} = (\oG_n^n)^{-1}\,\oG^n_m$, for $0\leqslant m\leqslant n-1$.

\subsection{Orthogonal polynomial systems (OPS)}

Let $\dd \mu(x,y)$ be a normalized measure defined on a region $ \Omega \subset\mathbb{R}^2$ and we assume that it is positive definite, in the way that
\begin{align*}
\int_{ \Omega } p^2(x,y) \, \dd \mu(x,y) > 0, &&
p\in\Pi, && p\not\equiv 0 .
\end{align*}
Moreover, we suppose that the moments 
\begin{align}\label{moment2D}
\omega_{h,k} = \int_{ \Omega } x^h\,y^k \, \dd \mu(x,y) < +\infty, &&
h, k \geqslant 0,
\end{align}
exist, and 
\begin{align*}
\omega_{0,0} = \int_{ \Omega } \dd \mu(x,y) =1.
\end{align*}
As usual, we define the inner product
\begin{align*}
\langle p, q\rangle = \int_{ \Omega } p(x,y)\,q(x,y)\, \dd \mu(x,y), && p, q \in \Pi,
&& \text{with} && \langle 1,1\rangle = 1 .
\end{align*}
We will see how the inner product acts over polynomial matrices. Let
$\oA
= \begin{bmatrix} a_{i,j}(x,y) \end{bmatrix}_{i,j=1}^{h,k}$ 
and $\oB = \begin{bmatrix} b_{i,j}(x,y) \end{bmatrix}_{i,j=1}^{k,l} $ 
be two polynomial matrices, i.e., $a_{i,j}(x,y), b_{i,j}(x,y) \in\Pi$.
The action of the above inner product over polynomial matrices is
defined as the $h \times l$ matrix (cf.~\cite{DX14}),
\begin{align*}
\langle \oA, \oB\rangle = \int_{ \Omega } \oA(x,y)\, \oB(x,y)\, \dd \mu(x,y) =
\begin{bmatrix}
\displaystyle
\int_{ \Omega } c_{i,j}(x,y) \dd\mu(x,y)
\end{bmatrix}_{i,j=1}^{h,l},
\end{align*}
where
$\displaystyle \oC = \oA \cdot \oB = \begin{bmatrix} c_{i,j}(x,y) \end{bmatrix}_{i,j=1}^{h,l}$.
We say that $\big\{\mP _n\big\}_{n\geqslant 0}$ is an
\emph{orthogonal polynomial system (OPS) with respect to $\langle\cdot, \cdot\rangle$}~if
\begin{align}\label{matrixHn}
\langle \mP _n, \mP ^{\top}_m \rangle
=
\begin{cases}
\mathtt{0}_{(n+1) \times (m+1)}, & n\ne m, \\
\oH_n, & n=m,
\end{cases}
\end{align}
where $\oH_n$ is a symmetric and positive-definite matrix of size $n+1$, and $\mathtt{0}_{(n+1) \times (m+1)}$, or $\mathtt{0}$ for short, is the zero matrix of adequate size. When
$\oH_n$ is diagonal, 
$n \in \mathbb N$, we say
that $\big\{\mP _n\big\}_{n\geqslant 0}$ is a \emph{mutually orthogonal polynomial system}. 
Moreover, there exists an \emph{orthonormal PS} satisfying~\eqref{matrixHn} with $\oH_n= \oI_{n+1}$. In addition, there exists a unique monic orthogonal polynomial system associated to $\dd \mu(x,y)$.

 \subsection{Three term relations}

Orthogonal polynomials in two variables satisfy, in each variable, a three term relation. These three
term relations are written in a vector form and have matrix coefficients.

\begin{proposition}[cf.~\cite{DX14}]
Let $\big\{\mP _n\big\}_{n\geqslant 0}$ be a monic OPS associated to an inner product $\langle \cdot, \cdot\rangle$. 
For $n\geqslant 0$, there exist constant matrices $\oD_{n,i}$ and~$\oC_{n,i}$ of sizes respectively given~by $(n+1)\times(n+1)$ and $(n+1)\times n$,\,$i=1,2$, such~that
\begin{align}\label{monic_3tr}
\begin{cases}
 x \, \mP _n = \oL_{n,1} \, \mP _{n+1} + \oD_{n,1}\,\mP _n + \oC_{n,1}\,\mP_{n-1}, \\
 y\,\mP _n =\oL_{n,2} \, \mP _{n+1} + \oD_{n,2}\,\mP _n + \oC_{n,2}\,\mP_{n-1},
\end{cases}
\end{align}
where $\mP _{-1}=0$, $\oC_{-1,i}=0$, and
\begin{align*}
 \oD_{n,1}\, \oH_{n} = \langle x\,\mP _n,\,\mP _{n}^{\top}\rangle, &&
 \oD_{n,2}\, \oH_{n} = \langle y\,\mP _n,\,\mP _{n}^{\top}\rangle, \\
 \oC_{n,1}\, \oH_{n-1} = \oH_n\, \oL_{n-1,1}^{\top}, && \oC_{n,2}\, \oH_{n-1} = \oH_n\, \oL_{n-1,2}^{\top}.
\end{align*}
Moreover, the rank~conditions
\begin{align*}
\rank \oC_{n,i} = n, \quad i=1, 2, &&
\rank \begin{bmatrix} \oC_{n,1} & \oC_{n,2} \end{bmatrix} = n+1,
\end{align*}
hold.
\end{proposition}

On the other hand, we can add the relations in~\eqref{monic_3tr}, to get
\begin{align} \label{eq:soma}
(x+y)\,\mP _n = \oL_{n}\,\mP _{n+1} + \oD_{n}\, \mP _n + \oC_{n}\,\mP _{n-1},
\end{align}
where 
\begin{align*}
\oL_n & = \oL_{n,1} + \oL_{n,2} \,\in\, \mathcal{M}_{(n+1)\times(n+2)}(\mathbb{R}),\\
\oD_n & = \oD_{n,1} + \oD_{n,2} \,\in\, \mathcal{M}_{(n+1)\times(n+1)}(\mathbb{R}),\\
\oC_n & = \oC_{n,1} + \oC_{n,2}\,\in\, \mathcal{M}_{(n+1)\times n}(\mathbb{R}).
\end{align*}
Here we pointed out that, taking $n=0$ in~\eqref{monic_3tr}, and multiplying 
by~$\mP _0=1$, we get
\begin{align*}
\langle x\mP _{0}, \mP _{0}\rangle &= \oL_{0,1}\langle\mP _{1},\mP _{0}\rangle 
+ \oD_{0,1}\langle\mP _0,\mP _{0}\rangle,\\
\langle y\mP _{0}, \mP _{0}\rangle &= \oL_{0,2}\langle\mP _{1},\mP _{0}\rangle 
+ \oD_{0,2}\langle\mP _0,\mP _{0}\rangle,
\end{align*}
and then 
\begin{align*}
\omega_{1,0} = \oD_{0,1}, &&
\omega_{0,1} = \oD_{0,2}.
\end{align*}
since $\langle\mP _0,\mP _{0}\rangle = \langle 1,1\rangle= \omega_{0,0}=1$.

For the orthonormal polynomials $\big\{\widetilde{\mP}_n\big\}_{n \geqslant 0 }$, such that $\oH_n = \oI_{n+1}$, the three term relations~\eqref{monic_3tr} take a simpler~form, as we can see in the next~result.

\begin{proposition}[\cite{DX14}]
For $n\geqslant 0$, there exist real matrices $\oA_{n,i}$,~$\oB_{n,i}$ of respective sizes $(n+1)\times (n+2)$ and 
$(n+1)\times (n+1)$ such~that 
\begin{align}\label{normal_3tr}
\begin{aligned}
& x\,\widetilde{\mP} _n = \oA_{n,1}\,\widetilde{\mP} _{n+1} 
+ \oB_{n,1}\,\widetilde{\mP} _n + \oA^\T_{n-1,1}\,\widetilde{\mP} _{n-1},\\
& y\,\widetilde{\mP} _n = \oA_{n,2}\,\widetilde{\mP} _{n+1} 
+ \oB_{n,2}\,\widetilde{\mP} _n + \oA^\T_{n-1,2}\,\widetilde{\mP} _{n-1},
\end{aligned}
\end{align}
where we define $\widetilde{\mP}_{-1} =0$ and $\oA_{-1,i}=0$. Moreover, each $\oB_{n,i}$ is symmetric.
\end{proposition}
Following~\cite[p. 64]{DX14}, given an OPS $\big\{\mP_n\big\}_{n\geqslant 0}$ we can construct an orthonormal PS in the form
\begin{align*}
\widetilde{\mP} _n = \oH_n^{-1/2} \, \mP_n,
\end{align*}
where $\oH_n = \langle \mP _n, \mP ^{\top}_n \rangle$, and $\oH_n^{-1/2} = (\oH_n^{1/2})^{-1}$ is the inverse of the 
positive \textit{square root} of $\oH_n$, the unique positive-definite matrix satisfying $\oH_n^{1/2}\,\oH_n^{1/2} = \oH_n$ (cf.~\cite{HJ85}).
 In such a case, we get
\begin{align*}
\oA_{n,i} = \oH_n^{-1/2}\, \oL_{n,i}\, \oH_{n+1}^{1/2}, && 
\oB_{n,i} = \oH_n^{-1/2}\, \oD_{n,i}\, \oH_n^{1/2}.
\end{align*}

\subsection{Block Jacobi matrices}

As in the one variable case, we can consider in the two variables 
one, two infinite tridiagonal block matrices~$\pJ_i$, $i=1,2$, defined as~follows 
\begin{align}\label{Jacobi-block}
\pJ_i = 
\begin{bmatrix}
\oD_{0,i} & \oL_{0,i} & & & 
 \\
\oC_{1,i} & \oD_{1,i} & \oL_{1,i} & & 
 \\
 & \oC_{2,i} & \oD_{2,i} & 
 \oL_{2,i} & 
 \\
 & & \ddots & \ddots & \ddots 
\end{bmatrix},
\end{align}
where the matrices in the main diagonal $\oD_{n,i}$ are square matrices of increasing size $n+1$, the matrices in the subdiagonal, $\oC_{n,i}$, are $(n+1)\times n$ matrices, and in the superdiagonal, the matrices~$\oL_{n,i}$ are $n\times (n+1)$ matrices, and the other elements are zero matrices of appropriate~size.

Defining the infinite column vector 
\begin{align*}
\pP = \begin{bmatrix}
\mP ^{\top}_0 & \mP ^{\top}_{1} & \cdots
& \mP ^{\top}_{n} & \cdots
\end{bmatrix}^{\top} ,
\end{align*}
the three term relations~\eqref{monic_3tr} can be written as
\begin{align*}
x \, \pP = \pJ_1\,\pP, &&
 y\,\pP = \pJ_2\,\pP.
\end{align*}
Directly from the three term relation we have a formal result above the commutativity of the matrix entries of 
both block matrices.

\begin{lemma}
For $n\geqslant 0$, we get
\begin{align}\label{C_D_Symmetry}
\begin{aligned} 
\phantom{ola}&\oC_{n,1} \oC_{n-1,2} = \oC_{n,2} \oC_{n-1,1},
 \\
\phantom{ola}&\oD_{n,1}\,\oC_{n,2} + \oC_{n,1}\,\oD_{n-1,2}
= \oD_{n,2}\,\oC_{n,1} + \oC_{n,2}\,\oD_{n-1,1}, 
 \\
\phantom{ola}&\oL_{n,1} \oC_{n+1,2} + \oD_{n,1} \oD_{n,2} + \oC_{n,1} \oL_{n-1,2} 
 \\
 & \phantom{olaolaolaola} = \oL_{n,2} \oC_{n+1,1}+ \oD_{n,2} \oD_{n,1} + \oC_{n,2} \oL_{n-1,1},\\
\phantom{ola}& \oL_{n,1} \oD_{n+1,2} + \oD_{n,1} \oL_{n,2} = \oL_{n,2} \oD_{n+1,1} + \oD_{n,2} \oL_{n,1},
\end{aligned}
\end{align}
where we assume $ \oC_{0,1} = \oC_{0,2} = 0$.
\end{lemma}

\begin{proof}
Since $y \, x\,\mP _n = x\,y \,\mP _n$, from~\eqref{monic_3tr} we get
\begin{multline*}
 \oL_{n,1} \big( \oL_{n+1,2} \, \mP _{n+2} + \oD_{n+1,2}\,\mP _{n+1} 
 + \oC_{n+1,2}\,\mP_{n} ) \\
 + \oD_{n,1} \big( \oL_{n,2} \, \mP _{n+1} + \oD_{n,2}\,\mP _n + \oC_{n,2}\,\mP_{n-1} \big) \\
+ \oC_{n,1} \big( \oL_{n-1,2} \, \mP _{n} + \oD_{n-1,2}\,\mP_{n-1} + \oC_{n-1,2}\,\mP_{n-2} \big)
\\
 =\oL_{n,2} \big( \oL_{n+1,1} \, \mP _{n+2} + \oD_{n+1,1}\,\mP_{n+1} + \oC_{n+1,1}\,\mP_{n} \big) \\
 + \oD_{n,2} \big( \oL_{n,1} \, \mP _{n+1} + \oD_{n,1}\,\mP_{n} 
 + \oC_{n,1}\,\mP_{n-1} \big) \\
 + \oC_{n,2} \big( \oL_{n-1,1} \, \mP _{n} + \oD_{n-1,1}\,\mP_{n-1} + \oC_{n-1,1}\,\mP_{n-2} \big).
\end{multline*}
Taking $\langle y\,x \,\mP_n, \mP_{j}^\top \rangle = \langle x\,y\, \mP_n, \mP_{j}^\top \rangle$, successively for $j=n-2$,$n-1$,$n$,
$n+1$,$n+2$ we get the desired result.
\end{proof}

As it is shown in \cite[pp. 83 and seq.]{DX14}, above conditions are not enough to assure the commutativity of $ \pJ_1$ 
and $\pJ_2$ since they are operators. The results about commutativity are showed for the Jacobi matrix for orthonormal
polynomial systems. We consider the Jacobi block matrix based on the three term relations for the orthonormal polynomials~\eqref{normal_3tr}:
\begin{align}\label{Jacobi-block-normal}
\pL_i = 
\begin{bmatrix}
\oB_{0,i} & \oA_{0,i} & & 
 \\
\oA^\T_{0,i} & \oB_{1,i} & \oA_{1,i} & 
 \\
 & \oA^\T_{1,i} & \oB_{2,i} & \oA_{2,i} 
 \\
 & & \ddots & \ddots & \ddots
\end{bmatrix}, && i =1,2 ,
\end{align}
where the matrices $\oA_{n,i}$ and $\oB_{n,i}$ are given in the three term relations, and the other elements are zero matrices of appropriate size. Then the three term relations~\eqref{normal_3tr} can be written as
\begin{align*}
 x \, \widetilde{\pP} = \pL_1\,\widetilde{\pP}, &&
 y \, \widetilde{\pP} = \pL_2\,\widetilde{\pP},
\end{align*}
where $\widetilde{\pP} = \begin{bmatrix}
\widetilde{\mP} ^{\top}_0 & \widetilde{\mP} ^{\top}_{1} & \cdots & \widetilde{\mP} ^{\top}_{n} & \cdots
\end{bmatrix}^{\top}$ is the vector of the orthonormal polynomials.

The matrices $\pL_1$ and $\pL_2$ can be considered as linear operators which act
via matrix multiplication on $\ell^2$, where the domain of the operator consists of all sequences
in $\ell^2$ for which matrix multiplication yields sequences in $\ell^2$. As linear operators,
the matrices $\pL_1$ and $\pL_2$, does not need to be bounded, however, according to \cite[pp. 83 and seq.]{DX14}, a sufficient condition for the boundedness of $\pL_1$ and $\pL_2$ is the compactness of the support $\Omega$ of de measure~$\mu$ having the sequence $\left\{\widetilde{\mP} _{n}\right\}_{n\geqslant 0 }$ as orthonormal polynomials.

Now, defining
\begin{align*}
\pH = 
\begin{bmatrix}
\oH_{0} & & & \\
 & \oH_{1} & & \\
 & & 
 \ddots 
\end{bmatrix}, && \pH^{1/2} = 
\begin{bmatrix}
\oH_{0}^{1/2} & & & \\
 & \oH_{1}^{1/2} & & \\
 & & 
 \ddots 
\end{bmatrix},
\end{align*}
and $\pH^{-1}$, $\pH^{-1/2} 
 $, their respective inverse matrices, we have,
\begin{align*}
\pJ_i = \pH^{1/2}\,\pL_i\,\pH^{-1/2}, && \pL_i = \pH^{-1/2}\,\pJ_i\,\pH^{1/2}, && i=1,2.
\end{align*}
Following Lemma 3.4.4 in \cite[p. 85]{DX14}, 
and supposing that $\pL_i$ are bounded, then they
are self-adjoint and
commute. In this way,
\begin{align*}
\pJ_1 \,\pJ_2 &= \pH^{1/2}\,\pL_1\,\pH^{-1/2}\pH^{1/2}\,\pL_2\,\pH^{-1/2} 
= \pH^{1/2}\,\pL_1\,\pL_2\,\pH^{-1/2} \\
 & = \pH^{1/2}\,\pL_2\,\pL_1\,\pH^{-1/2} 
 = \pH^{1/2}\,\pL_2\,\pH^{-1/2}\pH^{1/2}\,\pL_1\,\pH^{-1/2} = \pJ_2 \,\pJ_1.
\end{align*}
We can summarize this in the following result.

\begin{lemma}
If
$\pL_i = \pH^{-1/2}\,\pJ_i\,\pH^{1/2}$, $i=1,2$, are bounded matrices,~then 
\begin{align*}
\pJ_1 \,\pJ_2 = \pJ_2 \,\pJ_1.
\end{align*}
\end{lemma}

In general, for the bounded case we have,
\begin{align*}
\pJ_1^h \, \pJ_2^k = \pJ_2^k \, \pJ_1^h, && h, k \geqslant 0,
\end{align*}
and the product $\pJ_1^h \, \pJ_2^k$ preserves the block structure described in~\eqref{Jacobi-block}, but not the 
zero block matrices. In particular, we write 
\begin{align*}
\pJ_1^h \, 
\pJ_2^k = 
\begin{bmatrix}
\oD^{(h,k)}_{0,0} & \oL^{(h,k)}_{0,1} & (*) & \cdots 
 \\
\oC^{(h,k)}_{1,0} & \oD^{(h,k)}_{1,1} & \oL^{(h,k)}_{1,2} & \ddots 
 \\
(*) & \oC^{(h,k)}_{2,1} & \oD^{(h,k)}_{2,2} & \ddots 
 \\
 \vdots & 
 \ddots & \ddots & \ddots
\end{bmatrix},
\end{align*}
where the $(*)$ denotes block matrices of adequate size not necessarily zero. For $h, k, n \geqslant 0 $, the respective 
size of the involved matrices are $(n+1)\times(n+1)$ for $\oD^{(h,k)}_{n,n}$, $(n+2)\times (n+1)$ for 
$\oC^{(h,k)}_{n+1,n}$, and
$(n+1)\times(n+2)$ for $\oL^{(h,k)}_{n,n+1}$.

Observe that
\begin{align*} 
\oD^{(1,0)}_{n,n} = \oD_{n,1}, && 
\oC^{(1,0)}_{n,n-1} = \oC_{n,1}, && 
\oL^{(1,0)}_{n-1,n} = \oL_{n-1,1}, \\ 
\oD^{(0,1)}_{n,n} = \oD_{n,2}, && 
\oC^{(0,1)}_{n,n-1} = \oC_{n,2}, && 
\oL^{(0,1)}_{n-1,n} = \oL_{n-1,2}. 
\end{align*}
Let 
$\ell_0 
= \begin{bmatrix} 
1 & 0 & 0 & \cdots
\end{bmatrix}^{\top}$ 
be the zero column vector except for the first element. Then, it is clear that
\begin{align*}
\oD^{(h,k)}_{0,0} = \ell_0^{\top} 
\pJ_1^h \, 
\pJ_2^k \, \ell_0,
\end{align*}
i.e., we can access to the $(0,0)$ element of the matrix $ 
\pJ_1^h \, 
\pJ_2^k$.

Next result shows that all the moments are determined by the powers of the Jacobi matrices.

\begin{proposition}
For $h, k \geqslant 0 $,
\begin{align}\label{moments}
\omega_{h,k} = \ell_0^{\top} \pJ_1^h \, \pJ_2^k \, \ell_0.
\end{align}
\end{proposition}

\begin{proof}
Observe that, for all $h,k \geqslant 0$, we have
\begin{align*}
x^h y^k \,\pP = 
\pJ_1^h \, 
\pJ_2^k \,\pP,
\end{align*}
and then, the first element in both sides are equal,
\begin{align*}
x^h y^k \,\mP _0 = \oD_{0,0}^{(h,k)}\, \mP _0 + \oL^{(h,k)}_{0,1} \,\mP _1 + (*) \,\mP _2 + \cdots
\end{align*}
Integrating by means of the measure $\dd \mu(x,y)$, using the orthogonality, and the fact that $\mP _0=1$, we arrive~to 
$\displaystyle
\displaystyle \omega_{h,k} = \oD_{0,0}^{(h,k)} \omega_{0,0} = \oD_{0,0}^{(h,k)}$.
\end{proof}

\subsection{Stieltjes function in two variables}

From now on, we suppose that the positive definite measure is defined on a compact domain
$\Omega\subset \mathbb{R}^2$ and 
is given in terms of a weight function in the~form:
\begin{align*}
\dd \mu(x,y) = \omega(x,y) \dd x \dd y.
\end{align*}
Suppose that all the moments in~\eqref{moment2D} exist. 
Following~\cite{AFPP08}, we define the Stieltjes function in two variables as a double formal power series.

\begin{definition} 
The \emph{Stieltjes function} associated with a weight function $\omega(x,y)$ (or, equivalently, the operators $\pJ_1$ and $\pJ_2$)
is defined as 
\begin{align}\label{St}
\oS(z_1,z_2) = \int_{ \Omega } \frac{\omega(x,y)}{(z_1-x)\, (z_2-y)} \dd x \dd y.
\end{align}
\end{definition}

Observe that, substituting the formal series (cf.~\cite[p. 134]{Ca63})
\begin{align*}
\frac{1}{(z_1-x)\, (z_2-y)} = \sum_{h,k \geqslant 0 } \, \frac{x^h\,y^k}{z_1^{h+1}z_2^{k+1}},
\end{align*}
defined in 
$\big\{ (z_1,z_2) \in \mathbb C^2: |z_1| > \interleave \pJ_1 \interleave \mbox{ and } \ |z_2|> \interleave \pJ_2 \interleave \big\}$, where by $\interleave \pJ_i \interleave$ we mean the operator norm of $\pJ_i$, $i =1,2$,
we get
\begin{align*}
\oS(z_1,z_2) = \sum_{h,k \geqslant 0 } \frac{1}{z_1^{h+1}z_2^{k+1}}\int_{ \Omega } x^h\,y^k\,\omega(x,y) \dd x \dd y 
 =
\sum_{h,k \geqslant 0 } \, \frac{\omega_{h,k}}{z_1^{h+1}z_2^{k+1}}.
\end{align*}
We introduce the \emph{first $k$-marginal series} for the Stieltjes function~(\ref{St}),
\begin{align}\label{St-first}
\oS_{1,k}(z) = \sum_{h=0}^{+\infty} \, \frac{\omega_{h,k}}{z^{h+1}}, && k\geqslant 0,
\end{align}
and the \emph{second $h$-marginal series},
\begin{align}\label{St-second}
\oS_{2,h}(z) = \sum_{k=0}^{+\infty} \, \frac{\omega_{h,k}}{z^{k+1}}, && h\geqslant 0.
\end{align}
We observe that the functions
$\oS_{1,k}(z)$, $\oS_{2,h}(z)$,
are Stieltjes function in one variable associated with the moments
\begin{align*}
\int_{ \Omega } x^h \,y^k\,\omega(x,y) \dd x \dd y = \omega_{h,k}, 
 &&
h , k \geqslant 0 .
\end{align*}

\section{The 2D Toda lattices}\label{sec:3}

Now, we will show that the bivariate orthogonal polynomials modelize the $2$D Toda lattice. We suppose that the positive definite measure defined on a region $\Omega \subset \mathbb{R}^2$ is given in terms of a weight function in the~form:
\begin{align*}
\dd \mu(x,y) = \omega (x,y) \dd x \dd y,
\end{align*}
and suppose that all the moments in~\eqref{moment2D} exist. Let $\big\{\mP _n \big\}_{n\geqslant 0}$ be the monic OPS associated to the measure 
$
\dd \mu(x,y) $.

From now on we consider a weight function, $\omega
$, with an
\emph{evolution term}
$t \geqslant 0$, i.e.
$\omega(x,y, t )$
such that $\omega(x,y,0) \equiv \omega(x,y)$. Observe that the explicit shape of the weight function depending on the evolution term $t$ is unknown.

We also suppose that the time dependent weight function $\omega(x,y, t)$ has finite moments,~i.e.
\begin{align*}
\omega_{h,k}(t)
=\int_{ \Omega } x^h\,y^k
\, \omega(x,y, t ) \dd x \dd y < +\infty ,
 &&
h, k \geqslant 0 ,
 &&
t\geqslant 0 ,
\end{align*}
exist,
and is normalized such that $ \omega_{0,0}(t) = 1$.

Observe that the moments depend on the time variable $t$ and verify 
$\omega_{h,k}(0) = \omega_{h,k}$. For $t \geqslant 0$, we define the inner product,
\begin{align*}
\langle p, q\rangle_{t} 
= \int_{ \Omega } p(x,y)\,q(x,y)\,\omega(x,y, t ) \dd x \dd y, && p, q \in \Pi.
\end{align*}
Obviously, $\langle \pmb \cdot, \pmb \cdot \rangle_0 = \langle \pmb \cdot, \pmb \cdot \rangle$.

Let $\big\{\mP _n(t)\big\}_{n\geqslant 0}$ $\equiv$ $\big\{\mP _n(x,y,t)\big\}_{n\geqslant 0}$ be the monic orthogonal 
polynomial system associated to $\omega(x,y, t)$. Clearly, 
these are two variables polynomials (on $x$ and $y$) whose coefficients depend on $t$, 
\begin{align*}
\mP_n(t) = \mathbb{X}_n + \oG^{n}_{n-1}(t)\,\mathbb{X}_{n-1} + \cdots + \oG^{n}_0(t)\,\mathbb{X}_0,
 && n \geqslant 0,
\end{align*}
with $\mP _n (0) =\mP _n$.

As usual, we use the \textit{dot}-notation for the derivative with respect to the time variable~$ t $. Since~$\mP _n(t)$ is a monic polynomial we get 
\begin{align*}
\dt{\mP }_n(t) = \frac{\dd}{\dd t} \mP _n(t) = \dt{\oG}^{n}_{n-1}(t)\,\mathbb{X}_{n-1} + \dt{\oG}^{n}_{n-2}(t)\,\mathbb{X}_{n-2} +
\cdots + \dt{\oG}^{n}_0(t)\,\mathbb{X}_0,
\end{align*}
i.e., $\dt{\mP }_n(t)$ is a $(n+1)$ vector of polynomials of degree less than or equal to~$n-1$.

In addition, the symmetric positive-definite matrix
\begin{align*}
\oH_n(t) =\langle \mP _n(t), \mP ^{\top}_n(t)\rangle_t,
 && n \geqslant 0,
\end{align*}
also depends on $t $, and $\oH_n(0) = \oH_n$. 

The sequence $\{\mathbb {P}_n\}_{n\geqslant0}$ satisfy the three term relations~\eqref{monic_3tr}, but now, the matrix coefficients 
depend on $t $. In this way, for $n\geqslant 0$, there exist matrices
$\oD_{n,i}(t), \oC_{n,i}(t)$, of respective sizes $(n+1)\times (n+1)$ and $(n+1)\times n$, $i=1,2$, such~that
\begin{align}\label{t-3tr}
\begin{aligned}
 x\,\mP _n(t) = \oL_{n,1}\,\mP _{n+1}(t) + \oD_{n,1}(t)\, \mP _n(t) + \oC_{n,1}(t)\,
\mP _{n-1}(t), \\
 y\,\mP _n(t) = \oL_{n,2}\,\mP _{n+1}(t) + \oD_{n,2}(t)\, \mP _n(t) + \oC_{n,2}(t)\,
\mP _{n-1}(t),
\end{aligned}
\end{align}
where $\mP _{-1}(t) =0$ and $\oC _{-1}(t) =0$. Moreover,
\begin{align*}
\oD_{n,1}(t)\, \oH_{n}(t) & = \langle x\,\mP _n(t),\mP ^{\top}_{n}(t)\rangle_t, 
 &
\oD_{n,2}(t)\, \oH_{n}(t) 
 & = \langle y\,\mP _n(t),\mP ^{\top}_{n}(t)\rangle_t,
 \\
\oC_{n,1}(t)\, \oH_{n-1}(t) & = \oH_{n}(t)\, \oL^{\top}_{n-1,1}, &
\oC_{n,2}(t)\, \oH_{n-1}(t) & = \oH_{n}(t)\, \oL^{\top}_{n-1,2},
\end{align*}
and $\oD_{n,i}(0)= \oD_{n,i}$, $\oC_{n,i}(0) = \oC_{n,i}$, $i=1,2$. We must remark that the matrices $\oL_{n,i}$ defined 
in~\eqref{L_ni} are 
independent of $ t $. 

We also define 
\begin{align*}
\oD_n(t) & = \oD_{n,1}(t) + \oD_{n,2}(t), &&
\oC_n(t) = \oC_{n,1}(t) + \oC_{n,2}(t) .
\end{align*}
In this case, the tridiagonal Jacobi matrices also depend on the $t$ variable, and adopt the form
\begin{align}\label{Jacobi-block-t}
\pJ_i(t) = 
\begin{bmatrix}
\oD_{0,i}(t) & \oL_{0,i} & & 
 \\
\oC_{1,i}(t) & \oD_{1,i}(t) & \oL_{1,i} & 
 \\
 & \oC_{2,i}(t) & \oD_{2,i}(t) & \ddots 
 \\
 & & \ddots & \ddots
\end{bmatrix},
\end{align}
where the main diagonal matrices, $\oD_{n,i}(t)$, are of square type and with increasing size with $n$, the matrices in the subdiagonal, $\oC_{n,i}$, are $(n+1)\times n$ matrices, and in the superdiagonal, the matrices~$\oL_{n,i}$ are $n\times (n+1)$ matrices defined in \eqref{L_ni}, and the other elements are zero matrices of appropriate~size.

We define the infinite block matrix
\begin{align}\label{A-matrix}
\pA(t) = 
\begin{bmatrix}
\mathtt{0} & & & 
 \\
\oC_1(t) & \mathtt{0} & & 
 \\
 & \oC_2(t) & \mathtt{0} & 
 \\
 & & 
 \ddots & 
 \ddots 
\end{bmatrix},
\end{align}
where $\mathtt{0}$ are square zero matrices of adequate size. 

In the next theorem, we prove several characterizations for 2D Toda lattices, relating the orthogonality and the shape of the weight function with the evolution term. We will omit the $t$ variable when it can be understood by the context. 

\begin{theorem} \label{teo:main}
Let $w(x,y,t)$ be a normalized weight function defined on a domain $\Omega\subset \mathbb{R}^2$, and let $\big\{\mathbb{P}_n(t)\big\}_{n\geqslant 0}$ be the corresponding monic orthogonal polynomial system. The following statements are equivalent:

\begin{enumerate}[\rm (i)]

\item
The sequences $\big\{\oD_{n,i}(t)\big\}_{n \geqslant 0 }$ and $\big\{\oC_{n,i}(t)\big\}_{n\geqslant1}$, for $i=1,2$, satisfy the~$2$D Toda lattice
\begin{align}
\begin{cases}\label{CD-punto-i}
\dt{\oD}_{n,i}(t) = \oC_n(t)\,\oL_{n-1,i} - \oL_{n,i}\,\oC_{n+1}(t), 
 \\
\dt{\oC}_{n,i}(t) = \oC_{n}(t)\,\oD_{n-1,i}(t) - \oD_{n,i}(t)\, \oC_{n}(t), 
\end{cases} &&
n\geqslant 1 .
\end{align}

\item
The infinite matrices $\pJ_i(t)$ and $\pA(t)$ verify the Lax-Nakamura-type~pairs
\begin{align}\label{Lax}
\dt{ \pJ}_i = \big[\pA, \pJ_i \big] = \pA\, \pJ_i - \pJ_i\,\pA,
\end{align}
for $i=1,2$, where $ 
\pJ_i \equiv \pJ_i(t)$ and $ \pA_i \equiv 
\pA_i(t)$ are defined in~\eqref{Jacobi-block-t} and \eqref{A-matrix}, respectively.

\item
For $h,k\geqslant 0$, we have
\begin{align*}
& \dt{\aoverbrace[L1R]{\pJ_1^h\, \pJ_2^k}} = \pA\, \pJ_1^h\, \pJ_2^k - \pJ_1^h\,
\pJ_2^k\, \pA.
\end{align*}
\item For $h,k\geqslant 0$ the moments satisfy
\begin{align}\label{omega-dot}
\dt{\omega}_{h,k} = -\omega_{h+1,k} - \omega_{h,k+1} 
+ \omega_{h,k} \big( \omega_{1,0} + \omega_{0,1} \big).
\end{align}

\item
The Stieltjes function satisfies
\begin{align}\label{eq:eqdiffStiel}
\dt{\oS}(z_1,z_2, t ) = \big( \oD_0 - z_1 - z_2 \big) \oS(z_1,z_2, t ) + \oS_{1,0}(z_1, t ) + \oS_{2,0}(z_2, t ) .
\end{align}

\item
The weight function is solution of the differential equation
\begin{align}
\dt{\omega} = - \big( x_1 + x_2 - D_0 \big) \, \omega.
 \label{eq:funcional}
\end{align}

\item
The weight function is given by
\begin{align}\label{eq:representacao}
\omega(x,y,t) = \dfrac{\mathrm{e}^{-(x_1+x_2)t}\omega(x,y)}{\int_\Omega \mathrm{e}^{-(x_1+x_2)t}\omega(x,y)\dd x\dd y}.
\end{align}

\item
The MOPS satisfy
\begin{equation}\label{dot_P}
\dt{\mP}_n(t) = \oC_n(t) \,\mP_{n-1}(t),
\end{equation}
for $n\geqslant 1$ and $t\geqslant 0$.
\end{enumerate}
\end{theorem}

\begin{proof}
Rewriting~\eqref{CD-punto-i} in matrix notation we arrive to~\eqref{Lax}; and then, we get that, (i)~$\Rightarrow$ (ii).
Now, we will prove by induction that
(ii) $\Rightarrow$ (iii).
For $n=1$,~(ii) is the induction hypothesis. Applying an inductive reasoning we~get,
\begin{align*}
 \dt{\aoverbrace[L1R]{\pJ_i^n}} 
& = \dt{\aoverbrace[L1R]{\pJ_i^{n-1}}} \pJ_i + \pJ_i^{n-1}
 \dt{\pJ}_i\\
& = (\pA\,\pJ_i^{n-1} -\pJ_i^{n-1}\,\pA)\pJ_i + \pJ_i^{n-1}(\pA\,\pJ_i - \pJ_i\,\pA) && \text{by the hypotesis} \\
& = \pA\,\pJ_i^{n} - \pJ_i^n\,\pA , && i = 1,2 .
\end{align*}
Now, we can see that
\begin{align*}
\dt{\aoverbrace[L1R]{\pJ_1^h \pJ_2^k}} 
 & = 
 \dt{\pJ}_1^h \pJ_2^k
 + \pJ_1^h 
 \dt{\pJ}_2^k \\
 & = \big( \pA \pJ_1^h - \pJ_1^h \pA \big) \pJ_2^k + \pJ_1^h \big( \pA \pJ_2^k - \pJ_2^k \pA \big) && \text{by the last identity} \\
 & = \pA \big( \pJ_1^h \pJ_2^k \big) - \big( \pJ_1^h \pJ_2^k \big) \pA, && h,k \geqslant 0 .
\end{align*}
To prove that
(iii) $\Rightarrow$ (iv), we recall that the moments
\begin{align*}
\omega_{h,k}(t) = \int_{\Omega} x^h\,y^k\, \omega(x,y,t) \dd x \dd y
\end{align*}
depend on $t$, and by \eqref{moments}, are written in the form
\begin{align*}
\omega_{h,k}(t) = \ell_0^{\top} \big( \pJ_1(t) \big)^h \, \big(\pJ_2(t)\big)^k \ell_0.
\end{align*}
Taking derivatives and using (iii), we get
\begin{align*}
\dt{\omega}_{h,k} & = \ell_0^{\top} \,\dt{\aoverbrace[L1R]{ 
\pJ_1^h\,\ \pJ_2^k}} \,\ell_0 
= \ell_0^{\top} \, \big( \pA\, \pJ_1^h\, \pJ_2^k - \pJ_1^h\, \pJ_2^k\,\pA \big) \,\ell_0 = -\ell_0^{\top} \pJ_1^h\,
\pJ_2^k\,\ell_1\,\oC_1 \\
&= -\ell_0^{\top} 
\pJ_1^h\, \pJ_2^k\,\ell_0\, \oD_0-\ell_0^{\top} \pJ_1^h\,\pJ_2^k\,\ell_1\, \oC_1 + \ell_0^{\top} \pJ_1^h\, 
\pJ_2^k\,\ell_0\, \oD_0 \\
&= -\omega_{h+1,k} - \omega_{h,k+1} + \omega_{h,k} \oD_0 \\
& = -\omega_{h+1,k} - \omega_{h,k+1} + \omega_{h,k} \, \big( \omega_{1,0} + \omega_{0,1} \big),
\end{align*}
using the explicit expression of the matrix $\pA$ given in~\eqref{A-matrix} and the special shape of the 
matrix $ \pJ_1^h\, \pJ_2^k$. 
Here, 
$\ell_1 = 
\begin{bmatrix}
0 & 1 & 1 & 0 & \cdots 
\end{bmatrix}$, 
and together with~$\ell_0^{\top}$ allows access to the $(0,1)$ block matrix element in the matrix~$ 
\pJ_1^h\, \pJ_2^k$.

\noindent
From the definition of the Stieltjes function, and taking into account their time dependence, we can write using~(iv) that
\begin{align*}
\oS(z_1,z_2,t )
 =
\sum_{h,k=0}^{+\infty} \, \frac{\omega_{h,k}(t)}{z_1^{h+1}z_2^{k+1}} .
\end{align*}
Now, computing  the \emph{dot}-derivative, and applying~\eqref{omega-dot}, we successively get~that
\begin{align*}
& \dt{\oS}(z_1,z_2, t ) = \sum_{h,k=0}^{+\infty} \, \frac{\dt{\omega}_{h,k}}{z_1^{h+1}z_2^{k+1}} \\
& \hspace{.5cm}
= -\sum_{h,k=0}^{+\infty} \, \frac{\omega_{h+1,k}}{z_1^{h+1}z_2^{k+1}}
-\sum_{h,k=0}^{+\infty} \, \frac{\omega_{h,k+1}}{z_1^{h+1}z_2^{k+1}} 
+ \oD_0\,\sum_{h,k=0}^{+\infty} \, \frac{\omega_{h,k}}{z_1^{h+1}z_2^{k+1}}\\
 & \hspace{.5cm}
 = -\sum_{h=1,k=0}^{+\infty} \, \frac{\omega_{h,k}}{z_1^{h}z_2^{k+1}}
-\sum_{h=0,k=1}^{+\infty} \, \frac{\omega_{h,k}}{z_1^{h+1}z_2^{k}} 
+ \oD_0\,\oS(z_1, z_2, t )\\
 & \hspace{.5cm}
 = -z_1 \Big( \oS(z_1,z_2, t ) - \sum_{k=0}^{+\infty} \, \frac{\omega_{0,k}}{z_1 z_2^{k+1}}\Big) 
 -z_2\Big( \oS(z_1,z_2 , t ) - \sum_{h=0}^{+\infty} \, \frac{\omega_{h,0}}{z_1^{h+1}z_2}\Big)
 \\
 & \hspace{9.5cm}
 + \oD_0\,\oS(z_1, z_2, t )\\
 & \hspace{.5cm}
 = - \big( z_1+z_2- \oD_0 \big) \,\oS(z_1,z_2, t ) + \sum_{h=0}^{+\infty} \, \frac{\omega_{h,0}}{z_1^{h+1}} + \sum_{k=0}^{+\infty} \, \frac{\omega_{0,k}}{z_2^{k+1}}\\
 & \hspace{.5cm}
 = - \big( z_1+z_2- \oD_0 \big) \,\oS(z_1,z_2, t ) + \oS_{1,0}(z_1,t) + \oS_{2,0}(z_2, t ),
\end{align*}
by using the marginal Stieltjes functions defined in \eqref{St-first}-\eqref{St-second}.
There\-fore, the bi\-vari\-ate Stieltjes function $\oS(x,y, t )$ satisfies the \emph{dot}-dif\-fer\-en\-tial
equation~\eqref{eq:eqdiffStiel}, and so we arrive to~(v).
 \\[.105cm]
\noindent
Departing from (v), taking into account the representation~\eqref{St} for the Stieltjes function associated with $\omega$, and equation~\eqref{eq:eqdiffStiel}, we have
\begin{multline*}
\int_{\Omega} \frac{\dt{\omega} (x,y ,t)}{(z_1 - x) (z_2 - y)} \, \dd x \dd y
= - 
\int_{\Omega} \frac{(z_1 -x) \omega (x,y ,t)}{(z_1 - x) (z_2 - y)} \, \dd x \dd y
 \\ - 
\int_{\Omega} \frac{(z_2 -y) \omega (x,y ,t)}{(z_1 - x) (z_2 - y)} \, \dd x \dd y
 -
\int_{\Omega} \frac{(x + y - \oD_0) \omega (x,y ,t)}{(z_1 - x) (z_2 - y)} \, \dd x \dd y \\
 + \oS_1 (z_1,t) + \oS_2 ( z_2,t) .
\end{multline*}
Now, from~\eqref{St-first} and~\eqref{St-second} we see that the functions $\oS_1$ and $\oS_2$
have the following integral representation
\begin{align*}
\oS_1 (z_1 , t) & = \int_{\Omega} \frac{\omega (x,y ,t)}{z_1 - x} \, \dd x \dd y , &&
\oS_2 (z_1 , t) = \int_{\Omega} \frac{\omega (x,y ,t)}{z_2 - y} \, \dd x \dd y .
\end{align*}
Hence the above equation takes the form
\begin{align*}
\int_{\Omega} \frac{\dt\omega (x,y ,t) + (x + y - \oD_0)}{(z_1 - x) (z_2 - y)} \, \dd x \dd y
 = 0 ,
\end{align*}
and so we get~\eqref{eq:funcional} (by the analiticity of this function).
 \\[.105cm]
\noindent
Now, departing from~(vi) 
and solving~\eqref{eq:funcional}, we get
\begin{align*}
\omega(x,y,t) = \kappa^{-1} \mathrm{e}^{-(x_1+x_2 - \omega_{10}- \omega_{01})t }\omega(x,y) && \text{with} &&
\omega_{00} = 1 ,
\end{align*}
with $\kappa$ given by
$\displaystyle
\kappa = \int_\Omega \mathrm{e}^{-(x_1+x_2 - \omega_{10}- \omega_{01})t }\omega(x,y) \dd x \dd y $.
Using this expression for $\kappa$ we get the representation~\eqref{eq:representacao} for $\omega$.
 \\[.105cm]
\noindent
(vii) $\Rightarrow$ (viii):
As $\big\{ \mP_n \big\}_{n\geqslant0}$ is a basis of the linear space of two variable polynomials we know that, there exist matrices $\alpha_0, \alpha_1, \ldots , \alpha_{n-1}$ (that could depend on $t$, but not on $x,y$) such that
\begin{align*}
\dt{\mP}_n = \alpha_0 \,\mP_{n-1} + \alpha_1 \,\mP_{n-2} + \cdots + \alpha_{n-1} \,\mP_{0} .
\end{align*}
Using the orthogonality conditions we get that
\begin{align*}
\alpha_k H_{n-1-k} = \int_\Omega \dt{\mathbb P}_n \, \mathbb P_{n-1 -k}^\top \, \omega \, \dd x \dd y , && k \in \big\{ 0,1 , \ldots , n-1 \big\}.
\end{align*}
We determine the Fourier coefficients in the above expansion as
\begin{align*}
 \int_\Omega \dt{\mathbb P}_n \, \mathbb P_{n-1 -k}^\top \omega \dd x \dd y
 & =
 \int_\Omega {\mathbb P}_n \, \dt{\mathbb P}_{n-1 -k}^\top \omega \dd x \dd y 
 +
 \int_\Omega {\mathbb P}_n \, \mathbb P_{n-1 -k}^\top \, \dt \omega \dd x \dd y \\
 & = - \int_\Omega {\mathbb P}_n \, \mathbb P_{n-1 -k}^\top \, (x+y - D_0) \, \omega \, \dd x \dd y \\
 & = - \int_\Omega {\mathbb P}_n \, \mathbb P_{n-1 -k}^\top \, (x+y) \omega \, \dd x \dd y .
\end{align*}
Now, using the orthogonality we get that
$\alpha_k \equiv \mathtt{0}$, the null matrix, for $k = 1, \ldots , n-1$, and
from~\eqref{eq:soma}, i.e.
\begin{align*}
(x+y)\,\mP _{n} = \oL_{n}\,\mP _{n +1} + \oD_{n}\, \mP _{n} + \oC_{n}\,\mP _{n-1} ,
\end{align*}
we get that
\begin{align*}
\alpha_{0} H_{n-1} = \oC_{n}H_{n-1}, && \text{i.e.} && \alpha_{0} = \oC_{n} .
\end{align*}
From here the result follows.
 \\[.105cm]
\noindent
We close the diagram by differentiating the first relation in~\eqref{t-3tr}
\begin{align*}
x\,\dt{\mP} _n = \oL_{n,1}\,\dt{\mP} _{n+1} + \dt{\oD}_{n,1}\, \mP _n + \oD_{n,1}\, \dt{\mP} _n
+ \dt{\oC}_{n,1}\,\mP _{n-1} + \oC_{n,1}\,\dt{\mP} _{n-1},
\end{align*}
and using~\eqref{dot_P}, i.e., we apply the hypotesis (viii),
\begin{multline*}
x\,\oC_{n}\mP _{n-1} = \oL_{n,1}\,\oC_{n+1} \mP _{n} + \dt{\oD}_{n,1}\, \mP _n + \oD_{n,1}\, \oC_{n}\mP _{n-1}
 \\ + \dt{\oC}_{n,1}\,\mP _{n-1} + \oC_{n,1}\,\oC_{n-1}\mP _{n-2},
\end{multline*}
and again, we use~\eqref{t-3tr} for $n-1$, to deduce
\begin{multline*}
\oC_{n} (\oL_{n-1,1}\,\mP _{n} + \oD_{n-1,1}\, \mP _{n-1} + \oC_{n-1,1}\,\mP _{n-2}) \\
= (\oL_{n,1}\,\oC_{n+1} + \dt{\oD}_{n,1}) \mP _n 
+ (\oD_{n,1}\,\oC_{n} + \dt{\oC}_{n,1})\mP _{n-1} + \oC_{n,1}\,\oC_{n-1}\mP _{n-2}.
\end{multline*}
Since $\big\{\mP_n\big\}_{n\geqslant 0}$ is an orthogonal basis, matching the coefficients,~\eqref{CD-punto-i} follows for $i=1$. 

The case $i=2$ is analogous.
\end{proof}

As a corollary of this theorem we get a representation for the Stieltjes function associated with $\omega$ that modelizes the Toda lattice, as
\begin{multline*}
\oS(z_1,z_2, t ) 
= e^{-(z_1+z_2) t + \widetilde{d}(t)}
\Big( K - \int_0^{t} \big( \oS_{1,0}(z_1,s) \\
+ \oS_{2,0}(z_2,s) \big) e^{(z_1+z_2)s -\widetilde{d}(s)} \dd s \Big),
\end{multline*}
where $K$ is a positive constant and $\widetilde{d}(t) = \int_0^{t } \oD_0(s) \dd s$. 

Following~\cite{BP17}, we suppose that the matrix coefficients of the three term relation satisfy the~$2$D~Toda lattice
\begin{align}\label{C-punto-ib}
\begin{cases}
\dt{\oD}_{n,i}(t) = \oC_n(t)\, \oL_{n-1,i} - \oL_{n,i}\,\oC_{n+1}(t),\\
\dt{\oC}_{n,i}(t) = \oC_{n,i}(t)\,\oD_{n-1}(t) - \oD_{n}(t)\, \oC_{n,i}(t),
\end{cases}
\end{align}
for $i=1,2$, and $n\geqslant 0$, and then, summing both equations, we get the~system
\begin{align*}
\begin{cases}
\dt{\oD}_{n}(t) = \oC_{n}(t) \, \oL_{n-1} - \oL_{n} \, \oC_{n+1}(t), 
\\
\dt{\oC}_{n}(t) = \oC_{n}(t) \, \oD_{n-1}(t) - \oD_{n}(t) \, \oC_{n}(t).
\end{cases}
\end{align*}
Notice that, in equations~\eqref{C-punto-ib} appear $\oC_{n}(t)$ and $\oC_{n,i}(t)$. We prove that the Toda system deduced in \cite{BP17} is equivalent to our Toda system~\eqref{CD-punto-i}.

\begin{proposition}
The $2$D Toda lattice~\eqref{C-punto-ib} is equivalent to \eqref{CD-punto-i}.
\end{proposition} 

\begin{proof}
We transform~\eqref{C-punto-ib} by using~\eqref{C_D_Symmetry},
\begin{align*}
\dt{\oC}_{n,1}(t) & 
= \oC_{n,1}(t)\,\oD_{n-1}(t) - \oD_{n}(t)\, \oC_{n,1}(t) \\
 & 
 = \oC_{n,1}(t) (\oD_{n-1,1}(t) + \oD_{n-1,2}(t)) - (\oD_{n,1}(t) + \oD_{n,2}(t)) \oC_{n,1}(t)\\
 & = (\oC_{n,1}(t)+\oC_{n,2}(t)) \oD_{n-1,1}(t) - \oC_{n,2}(t) \oD_{n-1,1}(t) \\
 & \hspace{1.75cm}
+ \oC_{n,1}(t) \oD_{n-1,2}(t) - \oD_{n,1}(t)(\oC_{n,1}(t) + \oC_{n,2}(t)) \\
& \hspace{4.5cm} + \oD_{n,1}(t)\oC_{n,2}(t) - \oD_{n,2}(t) \oC_{n,1}(t)\\
 & = \oC_{n}(t)\, \oD_{n-1,1}(t) - \oD_{n,1}(t)\,\oC_{n}(t).
\end{align*}
In this way,~\eqref{CD-punto-i} is deduced. 
\end{proof}

\section{Isospectrality of $\pJ_i(t)$} \label{sec:4}

In this Section, we study the eigenvalues of the matrix $\pJ_i(t)$.

\begin{proposition}
The eigenvalues of the matrix $\pJ_i(t)$ are real, for $i=1,2.$
\end{proposition}

\begin{proof}
By \eqref{Jacobi-block-normal}, we get that
\begin{align*}
\pH^{1/2}\,\pL_i\,\pH^{-1/2} = \pJ_i, && i=1,2,
\end{align*}
i.e., $\pL_i$ and $\pJ_i$ are similar and have the same eigenvalues. Since $\pL_i$ is
self-adjoint
and real, then 
all of its eigenvalues are real, and then, the eigenvalues of $\pJ_i$ are.
\end{proof}

Next, we end the paper by proving that the matrices $\pJ_i(t)$ are isospectral, i.e., their eigenvalues are independent of the time 
variable~$t$, as occurs in the univariate case for standard Lax pairs (cf.~\cite{{La07}}).

\begin{theorem}
Let $\pJ_i(t)$ be given by~\eqref{Jacobi-block}.
Then, $\pJ_i(t)$, $i=1,2$, satisfies~\eqref{Lax} if, and only if, 
the algebraic spectrum of $\pJ_i(t)$, $i=1,2$, is independent of $t$.
\end{theorem}

\begin{proof}
From~\eqref{monic_3tr} written in matrix notation as
\begin{align*}
x_i \, \pP = \pJ_i\,\pP, && i =1,2 ,
\end{align*}
taking derivatives, and using~\eqref{Lax}
we get successively
\begin{align*}
\dt x_i \, \pP + x_i \, \dt \pP & = \big( \pA \, \pJ_i - \pJ_i \, \pA \big) \, \pP+ \pJ_i \, \dt \pP \\
\dt x_i \, \pP & = \big( \pJ_i - x_i \, \operatorname I \big) \, \big( - \pA \, \pP + \dt \pP \big) .
\end{align*}
Now, from Theorem~\ref{teo:main} we know that the
Lax{\textendash}Nakamura
 equation for $\pJ_i$, i.e.~\eqref{Lax}, is equivalent to~\eqref{dot_P},~i.e.
\begin{align*}
\dt{\mP}_n(t) = \oC_n(t) \,\mP_{n-1}(t) , && n \in \mathbb N ,
\end{align*}
which in matrix notation reads as
\begin{align*}
 \dt \pP
= \pA \, \pP , &&
\text{hence, we get that,}
&&
\dt x_i \, \pP = \operatorname 0,
\end{align*}
and so, $\displaystyle \dt x_i = 0$.
\end{proof}


\end{document}